\documentclass[12pt]{amsart}
\usepackage{latexsym, wasysym}
\usepackage{amsthm}
\usepackage{amsmath}
\usepackage{amsfonts}
\usepackage{amssymb}
\usepackage[dvips]{graphicx}
\usepackage{xy}
\xyoption{all}

\usepackage[active]{srcltx}  

\newtheorem{theo}{Theorem}[section]
\newtheorem{lemma}[theo]{Lemma}
\newtheorem{propo}[theo]{Proposition}
\newtheorem{defi}[theo]{Definition}

\newtheorem{rem}[theo]{Remark}
\newtheorem{exam}[theo]{Example}

\theoremstyle{definition}

\theoremstyle{remark}

\newcommand\colim{\mathop{\rm colim}\limits}

\newcommand\Loc{\operatorname{Loc}}
\newcommand\LEq{\operatorname{LEq}}

\newcommand\op{\operatorname{op}}
\newcommand\ev{\operatorname{ev}}

\newcommand\id{\operatorname{id}}
\newcommand\Id{\operatorname{Id}}

\newcommand\SSet{\operatorname{\bf SSet}}

\newcommand\cof{\operatorname{cof}}

\newcommand\ca{\mathcal {A}}

\newcommand\cg{\mathcal {G}}
\newcommand\cc{\mathcal {C}}
\newcommand\cd{\mathcal {D}}
\newcommand\cf{\mathcal {F}}
\newcommand\ch{\mathcal {H}}
\newcommand\ci{\mathcal {I}}

\newcommand\cj{\mathcal {J}}
\newcommand\ck{\mathcal {K}}
\newcommand\cl{\mathcal {L}}
\newcommand\cm{\mathcal {M}}
\newcommand\crr{\mathcal {R}}

\newcommand\ct{\mathcal {T}}
\newcommand\cp{\mathcal {P}}

\newcommand\cw{\mathcal {W}}

\newcommand\cal{\mathcal}
\newcommand\Ex{\operatorname{Ex}}
\newcommand\Fib{\operatorname{Fib}}

\newcommand\fin{\operatorname{fin}}

\DeclareMathOperator{\Hor}{\ensuremath{\textup{Hor}}}

\date{October 18, 2011}
  
\begin{document}
\title[Class-combinatorial model categories]
{Class-combinatorial model categories}
\author[B. Chorny and J. Rosick\'{y}]
{B. Chorny and J. Rosick\'{y}$^*$}
\thanks{ $^*$ Supported by MSM 0021622409 and GA\v CR 201/11/0528. The hospitality of the Australian National University 
is gratefully acknowledged.} 
\address{\newline B. Chorny\newline
Department of Mathematics\newline
University of Haifa at Oranim\newline
Tivon, Israel\newline
chorny@math.haifa.ac.il
\newline\newline 
J. Rosick\'{y}\newline
Department of Mathematics and Statistics\newline
Masaryk University, Faculty of Sciences\newline
Brno, Czech Republic\newline
rosicky@math.muni.cz
}

\begin{abstract}
We extend the framework of combinatorial model categories, so that the category of small presheaves over large indexing categories and ind-categories 
would  be embraced by the new machinery called class-combinatorial model categories.

The definition of the new class of model categories is based on the corresponding extension of the theory of locally presentable and accessible 
categories developed in the companion paper \cite{CR}, where we introduced  the concepts of locally class-presentable and class-accessible categories.

In this work we prove that the category of weak equivalences of a nice class-combinatorial model category is class-accessible. Our extension of J.~Smith 
localization theorem depends on the verification of a cosolution-set condition. The deepest result is that the (left Bousfield) localization 
of a class-combinatorial model category with respect to a strongly class-accessible localization functor is class-combinatorial again.
\end{abstract}


\maketitle
 
\section{Introduction}
The theory of combinatorial model categories pioneered by J.~Smith in the end of '90-s has become a standard framework for abstract homotopy theory. 
The foundations of the subject may be found in \cite{B} and \cite{D-universal, D}; a concise exposition has appeared in \cite[A.2.6]{L}.

A model category $\cm$ is combinatorial if it satisfies two conditions. The first condition requires that the underlying category $\cm$ is locally 
presentable (see, e.g., \cite{AR} for the definition and an introduction to the subject). The second condition demands that the model structure 
will be cofibrantly generated (see, e.g., \cite{Ho} for the definition and discussion).

Several interesting examples of non-combinatorial model categories appeared over the past decade. For example the categories of pro-spaces and ind-spaces 
were applied in new contexts in homotopy theory \cite{CI,I} resulting in non-cofibrantly generated model structures constructed on non-locally presentable
categories. The maturation of the calculus of homotopy functors \cite{Goodwillie-CalculusIII} stimulated the development of the abstract homotopy theory 
of small functors over large categories \cite{CD} resulting in formulation of the basic ideas of Goodwillie calculus in the language of model categories 
\cite{BCR}. The model categories used for this purpose are also not cofibrantly generated and the underlying category of small functors from spaces 
to spaces (or spectra) is not a locally presentable category.

However, all the model categories from the examples above are \emph{class-cofibrantly} generated (except for the pro-categories, which are 
\emph{class-fibrantly} generated). This extension of the classical definition was introduced in \cite{Chorny-prospaces}, which in turn developed 
the ideas by E.~Dror Farjoun originated in the equivariant homotopy theory \cite{Farjoun}.

The purpose of the current paper is to develop a framework extending J.~Smith's combinatorial model categories, so that the model categories of small 
presheaves over large categories, ind-categories of model categories (the opposite categories of pro-categories) would become the examples of the newly 
defined \emph{class-combinatorial} model categories. The definition of the class-combinatorial model category consists, similarly to the combinatorial 
model category, of two conditions: the underlying category is required to be \emph{locally class-presentable} and the model structure must be 
class-cofibrantly generated. As we mentioned above, the second condition was studied in the earlier work \cite{Chorny-prospaces}, while 
the first condition relies on a concept of the locally class-presentable category, which was introduced and studied in the companion project \cite{CR}, 
which is a prerequisite for reading this paper.

The main results of our paper are generalizing the corresponding results about the combinatorial model categories. In Theorem~\ref{th2.10} we prove 
that the levelwise weak equivalences in the category of small presheaves form a class-accessible category (see \cite{CR} for the definition). 
In Remark~\ref{re2.12} we formulate the mild conditions, which guarantee that the class-combinatorial model category has the class-accessible subcategory 
of weak equivalences. Such class-combinatorial model categories are called nice in this paper.

The central result of J.~Smith's theory is the localization theorem, stating the existence of the (left Bousfield) localization of any combinatorial 
model category with respect to any \emph{set} of maps. After a brief discussion of construction of localization functors with respect to cone-coreflective
\emph{classes} of cofibrations with bounded presentability ranks of domains and codomains, we prove in Theorem~\ref{th3.10} a variant of a localization 
theorem for nice class-combinatorial model categories with respect to strongly class-accessible homotopy localization functors (i.e., localization 
functors preserving $\lambda$-filtered colimits and $\lambda$-presentable objects for some cardinal $\lambda$).  Although an application 
of our localization theorem depends on the verification of a cosolution-set condition for the class of intended generating trivial cofibrations, 
we are able to check this condition in many interesting situations. In the last Theorem~\ref{th3.14} we prove that in the cases where the localization 
with respect to a strongly class-accessible functor exists, the localized model category is  class-combinatorial again. We conclude the paper by several 
examples of localized model categories. Using Theorem~\ref{th3.14} we show that the $n$-polynomial model category constructed in \cite{BCR} 
is class-combinatorial (Example~\ref{ex3.16}). On the other hand, there is a model category constructed in \cite{C1} as a localization 
of a class-combinatorial model category with respect to an inaccessible localization functor that happens to be non-cofibrantly generated 
(Example~\ref{ex3.17}).

\section{Class-combinatorial model categories} 
 
Recall that a weak factorization system $(\cl,\crr)$ in a locally class-$\lambda$-presentable category $\ck$ was called
\textit{cofibrantly class}-$\lambda$-\textit{generated} in \cite{CR} 4.7 if $\cl=\cof(\cc)$ for a cone-coreflective class $\cc$ of morphisms 
such that
\begin{enumerate}
\item[(1)] morphisms from $\cc$ have $\lambda$-presentable domains and codomains and
\item[(2)] any morphism between $\lambda$-presentable objects has a weak factorization with the middle object $\lambda$-presentable.
\end{enumerate}
To be \textit{cone-coreflective} means for each $f$ there is a subset $\cc_f$ of $\cc$ such that each morphism $g\to f$ in $\ck^\to$
with $g\in\cc$ factorizes as
$$
g\to h\to f
$$
with $h\in\cc_f$. 

If the weak factorization is functorial, a cofibrantly class-$\lambda$-generated weak factorization system is cofibrantly class-$\mu$-generated 
for each regular cardinal $\mu\trianglerighteq \lambda$. Without functoriality, condition (2) does not need to go up to $\mu$ and thus we will make it
a part of the following definition.

\begin{defi}\label{def2.1}
{\em
Let $\ck$ be a class locally-$\lambda$-presentable model category. We say that $\ck$ is \textit{class}-$\lambda$-\textit{combinatorial} 
if both (cofibrations, trivial fibrations) and (trivial cofibrations, fibrations) are cofibrantly class-$\mu$-generated weak factorization 
systems for every regular cardinal $\mu\trianglerighteq \lambda$.

It is called \textit{class-combinatorial} if it is class-$\lambda$-combinatorial for some regular cardinal $\lambda$.
}
\end{defi}

Any combinatorial model category is class-combinatorial. The reason is that weak factorizations are functorial and, moreover, the resulting
functors are accessible. Thus they are strongly accessible and this property goes up for $\lambda\triangleleft\mu$ (cf. \cite{R1}).

\begin{exam}\label{ex2.2}
{
\em
Let $\SSet$ denote the category of simplicial sets. Given a simplicial category $\ca$, by abuse of notation, $\cp(\ca)$ 
will denote the category of \emph{small simplicial presheaves} on $\ca$. The objects are functors $\ca^{\op}\to\SSet$ which are 
small weighted colimits of simplicial representable functors (see \cite{DL}). In \cite{CR}, we used this notation for small presheaves
on a category $\ca$ but it will cause any misunderstanding. The simplicial category $\cp(\ca)$ is complete provided that $\ca$ is complete 
(see \cite{DL}); completeness is meant in the enriched sense (see, e.g., \cite{Bor} or \cite{K1}).
 
The category $\cp(\ca)$ is always class-finitely-accessible, because each small simplicial presheaf is a conical colimit of presheaves 
from $\cg=\{\hom_{\ca}(-,A)\otimes K \,|\, A\in\ca,\, K\in\SSet\}$. Therefore each small presheaf is a filtered colimit of finite colimits 
of elements of $\cg$. The elements of $\cg$ are, in turn, filtered colimits of the elements 
of $\cg_{\fin}=\{\hom(-,A)\otimes L \,|\, A\in\ca,\, L\in\SSet_{\fin})\}$ where $\SSet_{\fin}$ denotes the full
subcategory of $\SSet$ consisting of finitely presentable simplicial sets. Therefore, every small presheaf is a filtered colimit 
of finite colimits of the elements of $\cg_{\fin}$.
}
\end{exam}

We are going to show that, for a complete simplicial category $\ca$, the category $\cp(\ca)$ equipped with the projective model structure 
is class-combinatorial. We will need the following result.

\begin{lemma}\label{le2.3}
Let $\ca$ be a complete simplicial category and $\mu$ an uncountable regular cardinal. Then $\mu$-presentable objects in $\cp(\ca)$ are closed 
under finite weighted limits. 
\end{lemma}
\begin{proof}
A weight $W:\cd\to\SSet$ is finite if $\cd$ has finitely many objects and all its hom-objects $\cd(c,d)$ and all values $W(d)$ are finitely 
presentable simplicial sets (see \cite{K} 4.1.) Following \cite{K} 4.3, finite weighted limits can be constructed from finite conical limits
and cotensors with finitely presentable simplicial sets. Thus we have to show that $\mu$-presentable objects in $\cp(\ca)$ are closed under
these limits. Following \cite{Bor} 6.6.13 and 6.6.16, finite conical limits in $\cp(\ca)$ coincide with finite limits in the underlying category 
$\cp(\ca)_0$. This underlying category is a subcategory of the category $\cp(\ca_0)$ of small functors from \cite{CR} 2.2 (2) which is closed under 
limits and colimits. Thus $\mu$-presentable objects in $\cp(\ca)$ are $\mu$-presentable in $\cp(\ca_0)$. Since the latter objects are closed under 
finite limits (see \cite{KRV} 4.9), $\mu$-presentable objects in $\cp(\ca)$ are closed under finite conical limits. It remains to show that 
they are closed under cotensors with finitely presentable simplicial sets. The later are finite conical colimits of cotensors with $\Delta_n$, 
$n=1,2,\dots$ (see the proof of \cite{DL}, 5.2). Thus we have to show that $\mu$-presentable objects in $\cp(\ca)$ are closed under cotensors 
with $\Delta_n$'s.  

Let $H$ be $\mu$-presentable in $\cp(\ca)$. Since $H$ is a $\mu$-small colimit of tensors $H_i$ of finitely presentable simplicial sets
with representables and both colimits and cotensors in $\cp(\ca)$ are pointwise, we have
\begin{align*}
H^{\Delta_n}(A)&=(\colim H_i)^{\Delta_n}(A)=(\colim H_i(A))^{\Delta_n}\\
&=\hom(\Delta_n,\colim H_i(A))\cong\colim\hom(\Delta_n,H_i(A))\\
&=\colim H_i(A)^{\Delta_n}=(\colim H_i^{\Delta_n})(A)
\end{align*}
for each $A$ in $\ca$. Hence
$$
H^{\Delta_n}\cong\colim H_i^{\Delta_n}
$$
and thus it suffices to show that each $H_i^{\Delta_n}$ is $\mu$-presentable. Since each $H_i$ is equal to $V\otimes\hom(-,B)$ for
some finitely presentable simplical set $V$ and $B$ in $\ca$, we get for the same reasons as above
\begin{align*}
& \hom(\Delta_n,V \otimes \hom(A,B)) =
\hom(\Delta_n, V \times \hom(A,B)) \cong\\
& \hom(\Delta_n,V) \times \hom(\Delta_n,\hom(A,B)) \cong 
V^{\Delta_n} \times \hom(A,B^{\Delta_n}) =\\
& (V^{\Delta_n} \otimes \hom(-,B^{\Delta_n}))(A)
\end{align*}
and thus
$$
(V\otimes\hom(-,B))^{\Delta_n} \cong V^{\Delta_n} \otimes \hom(-,B^{\Delta_n}).
$$
The latter objects are $\mu$-presentable.
\end{proof}

\begin{propo}\label{prop2.4}
Let $\ca$ be a complete simplicial category. Then $\cp(\ca)$ is class-$\lambda$-combinatorial with respect to the projective model structure
for each uncountable regular cardinal $\lambda$.
\end{propo}
\begin{proof}
Following \cite{CD}, $\cp(\ca)$ is a model category where the generating classes $\ci$ and $\cj$ of cofibrations and trivial cofibrations
are cone-coreflective and satisfy condition \cite{CR} 4.7 (1) for any regular cardinal $\lambda$. In fact, $\ci$ consists of morphisms
$$
\partial\Delta_n\otimes\hom(-,A)\to\Delta_n\otimes\hom(-,A)
$$
and $\cj$ of morphisms
$$
\Lambda^k_n\otimes\hom(-,A)\to\Delta_n\otimes\hom(-,A),
$$
and all involved domains and codomains are finitely presentable. We have to show that they satisfy \cite{CR} 4.7 (2)  as well, i.e., that 
they are bounded. Let $\lambda$ be uncountable and $f:G\to H$ be a morphism between $\lambda$-presentable objects and consider a morphism 
$g\to f$ where $g\in\ci$. Following the proof of 3.7 in \cite{CD}, this morphism corresponds to a morphism $\hom(-,A)\to P$ where $P$ 
is the pullback 
$$
G^{\partial\Delta_n}\times_{H^{\partial\Delta_n}}H^{\Delta_n}.
$$
Since $P$ is $\lambda$-presentable (see \ref{le2.3}), there is a choice of a set $\ct_f$ from \cite{CR} 4.8 (2) whose cardinality 
does not exceed $\lambda$. Since all morphisms from $\ci$ have finitely presentable domains and codomains, the factorization of $f$ 
stops at $\omega$. Thus the cardinality of $\ct_f^\ast$ is smaller than $\lambda$. Following \cite{CR} 4.8 (2), condition \cite{CR} 
3.7 (2) is satisfied. The argument for $\cj$ is the same.
\end{proof}

\begin{rem}\label{re2.5}
{
\em
A very useful property of the combinatorial model categories is that the class of weak equivalences is an accessible and accessibly
embedded subcategory of the category of morphisms $\ck^\to$ (see \cite{R2} 4.1 or \cite{L} A2.6.6). Together with Smith's theorem \cite{B} 
it constitutes the localization theorem for combinatorial model categories with respect to sets of maps. It would be natural to expect that 
a similar property holds in the class-combinatorial situation. Unfortunately we were unable to prove it in this generality. But in many interesting 
situations we are able to prove that the class of weak equivalences is a class-accessible subcategory of the category of morphisms.  
}
\end{rem}

\begin{lemma}\label{le2.6}
Let $\ca$ be a complete simplicial category. Then $\cp(\ca)$ admits a strongly class-accessible fibrant replacement functor.
\end{lemma}
\begin{proof}
The functor $\Ex^\infty:\SSet\to\SSet$ is the finitely accessible fibrant replacement simplicial functor (see \cite{GJ}). For a small simplicial
functor $F:\ca^{\op}\to\SSet$, let $\Fib(F)$ be the composition
$$
\ca^{\op} \xrightarrow{\quad  F\quad} \SSet\xrightarrow{\quad Ex^\infty\quad} \SSet.
$$
We will show that this composition is small. 

The category of finitely accessible simplicial functors $\SSet\to\SSet$ is equivalent to the category of simplicial functors 
$\SSet^{\SSet_{\fin}}$. This equivalence sends a finitely accessible functor $\SSet\to\SSet$ to its restriction on $\SSet_{\fin}$.
Thus hom-functors $\hom(S,-):\SSet\to\SSet$ with $S$ finitely presentable correspond to hom-functors $\hom(S,-):\SSet_{\fin}\to\SSet$.  
Since every simplical functor $\SSet_{\fin}\to\SSet$ is a weighted colimit of hom-functors, every finitely accessible simplicial functor
$\SSet\to\SSet$ is a weighted colimit of hom-functors $\hom(S,-)$ with $S$ finitely presentable. Thus the composition $\Ex^\infty F$
is a weighted colimit of functors $\hom(S,-)F$ with $S$ finitely presentable. But the fun\-ctor $\hom(S,-)F$ is small because it is isomorphic 
to the cotensor $F^S$. The reason is that natural transformations 
$$
\ca(-,A)\to \hom(S,-)F=\hom(S,F-)
$$ 
correspond to morphisms $S\to FA$, i.e., to morphisms 
$$
S\to\cp(\ca)(\ca(-,A),F)
$$ 
which, by the definition of the cotensor, correspond to morphisms $\ca(-,A)\to F^S$. Consequently, $\Ex^\infty F$ is small as a weighted
colimit of small functors. 

We have obtained the functor $\Fib:\cp(\ca)\to\cp(\ca)$ which clearly has fibrant values. Moreover, the pointwise trivial cofibration
$$
\Id_{\SSet}\to\Ex^\infty
$$ 
yields a weak equivalence $\Id_{\cp(\ca)}\to\Fib$. Thus $\Fib$ is a fibrant replacement functor on $\cp(\ca)$. Since $\Ex^\infty$ is finitely 
accessible, so is $\Fib$. We know that $\Ex^\infty$ is a weighted colimit of hom-functors $\hom(S,-)$ with $S$ finitely presentable.
The corresponding weight is $\lambda$-small for an uncountable regular cardinal $\lambda$. Let $F$ be $\lambda$-presentable in $\cp(\ca)$.
Then $\Fib(F)$ is a $\lambda$-small weighted colimit of $\hom(S,-)F\cong F^S$ and the latter functors are $\lambda$-presentable following
\ref{le2.3}. Hence $\Fib(F)$ is $\lambda$-presentable (the argument is analogous to \cite{K}, 4.14). Thus $\Fib$ is strongly 
class-$\lambda$-accessible.
\end{proof}

\begin{defi}\label{def2.7}
{
\em
Let $\ca$ be a complete simplicial category and $f:A\to B$ be a morphism in $\cp(\ca)$. The \textit{Serre construction} on $f$ is the object
$S(f)$ of $\cp(\ca)$ defined as a pullback
$$
\xymatrix{
S(f) \ar [r]^{r} \ar [d]_{q}& B^{\Delta_1} \ar [d]^{B^j}\\
B\ar [r]_{f}& B
}
$$
where $j:\Delta_0\to\Delta_1$ sends $0$ to $0$.
} 
\end{defi}

\begin{rem}\label{re2.8}
{
\em
The Serre construction was used in the PhD thesis of J.P.~Serre in order to replace an arbitrary map of topological spaces by a fibration. 
We are going to use it pretty much for the same purpose in $\cp(\ca)$. The advantage over the modern methods of factorization is the functoriality 
of $S(-)$.
}
\end{rem}

\begin{lemma}\label{le2.9}
Let $\ca$ be a complete simplicial category and $f:A\to B$ a morphism of fibrant objects in $\cp(\ca)$. Then there exists a factorization 
$$
A \xrightarrow{\quad  i\quad} S(f)\xrightarrow{\quad p\quad} B
$$
of $f$ where $i$ is a weak equivalence and $p$ is a fibration.
\end{lemma}
\begin{proof}
The pullback in \ref{def2.7} may be split in two pulbacks
$$
\xymatrix@=3pc{
S(f) \ar[r]^{r} \ar[d]_{q_1} & B^{\Delta_1}\ar[d]^{B^u}\\
A\times B \ar[r]^{f\times\id_B} \ar[d]_{q_2} & B\times B
\ar[d]^{B^v}\\
A\ar[r]_{f} &B
}
$$
where 
$$
\Delta_0 \xrightarrow{\quad  v\quad} \Delta_0 +\Delta_0 \xrightarrow{\quad u\quad} \Delta_1
$$
is the factorization of $j$. Since both $u$ and $v$ are cofibrations and $B$ is fibrant, the vertical morphisms $B^u$ and $B^v$ are fibrations
(see \cite{H} 9.3.9 (2a)). Moreover, since $j$ is a trivial cofibration, $B^j$ is a trivial fibration. Thus $q_1$  and $q_2$ are fibrations
and $q=q_2g_1$ is a trivial fibration.

Let $t$ denote the unique morphism $\Delta_1\to\Delta_0$. Since,
$$
B^jf^{\Delta_1}A^t=fA^jA^t=f,
$$
there is a unique morphism $i:A\to S(f)$ such that $qi=\id_A$ and $ri=f^{\Delta_1}A^t$. Since $q$ is a trivial fibration, $i$ is a weak
equivalence. Since $B^v:B\times B\to B$ is the first projection of the product, $q_2:A\times B\to A$ is the first projection as well. Let
$p_2:A\times B\to B$, $\overline{p}_2:B\times B\to B$ be the second projections and $v':\Delta_0\to\Delta_0+\Delta_0$ be the second injection
of the coproduct. Then $\overline{p}_2=B^{v'}$ and
\begin{align*}
p_2q_1i&=\overline{p}_2(f\times\id_B)q_1i=\overline{p}_2B^uri=\overline{p}_2B^uf^{\Delta_1}A^t=B^{v'}B^uf^{\Delta_1}A^t\\
&=B^{uv'}f^{\Delta_1}A^t=fA^{uv'}A^t=f.
\end{align*}
Since $B$ is fibrant, $p_2$ is a fibration and thus $p=p_2q_1$ is a fibration. We have $f=pi$.
\end{proof}

\begin{theo}\label{th2.10}
Let $\ca$ be a complete simplicial category and denote by $\cw$ the class of weak equivalences in the projective model structure on $\cp(\ca)$. 
Then $\cw$ is a class-accessible category strongly accessibly embedded in $\cp(\ca)^\to$.
\end{theo}
\begin{proof}
Let $\Fib:\cp(\ca)\to\cp(\ca)$ be the strongly class-accessible fibrant replacement functor constructed in \ref{le2.6}. Consider the functor
$$
R:\cp(\ca)^\to\to\cp(\ca)^\to
$$
assigning to a morphism $f:A\to B$ the fibration $p:S(\Fib(f))\to\Fib(B)$ from \ref{le2.9}. Since the construction of $S(f)$ uses only finite limits, 
the functor $S(-):\cp(\ca)^\to\to\cp(\ca)$ is strongly class-accessible by \ref{le2.3}. Therefore the functor $R(-)$ is also strongly class-accessible.
A morphism $\alpha:F\to G$ in $\cp(\ca)$ is a weak equivalence if and only if $\Fib(f)$ is a weak equivalence, i.e., if and only if $R(f)$ is a trivial 
fibration.  
 
Let $\cf_0$ denote the full subcategory of $\cp(\ca)^\to$ consisting of trivial fibrations. Following \ref{prop2.4} and \cite{CR} 4.9, $\cf_0$ is 
class-$\lambda$-accessible and strongly class-$\lambda$-accessibly embedded in $\cp(\ca)^\to$ for every uncountable regular cardinal $\lambda$.
Since $\cw$ is given by the pullback
$$
\xymatrix@=4pc{
\ck^\to \ar[r]^{R} & \ck^\to \\
\cw \ar[u] \ar[r] &
\cf_0 \ar[u]
}
$$
whose vertical leg on the right is transportable, $\cw$ is equivalent to the pseudopullback (see \cite{CR} 3.2). Thus \cite{CR} 3.1 implies 
that $\cw$ is a class-accessible subcategory of $\cp(\ca)^{\to}$.
\end{proof}

\begin{defi}\label{def2.11}
{
\em
A class-combinatorial model category $\ck$ is \textit{nice} if the class of weak equivalences $\cw$ is a class-accessible, strongly class-accessibly 
embedded subcategory of $\ck^{\to}$.
}
\end{defi}

\begin{rem}\label{re2.12}
{
\em
We have just proved that $\cp(\ca)$ equipped with the projective model structure is  a nice model category for any complete simplicial category $\ca$. 
The same argument applies to every simplicial class-combinatorial model category which is equipped with a strongly class-accessible fibrant replacement 
functor and whose $\mu$-presentable objects are closed under finite weighted limits for each $\mu\geq\lambda$ (where $\lambda$ is a cardinal). 
We are not aware of any example of a class-combinatorial model category, which would fail to be nice. 
}
\end{rem}

\begin{theo}\label{th2.13}
Let $\ck$ be a locally class-$\lambda$-presentable category, $\ci$ a $\lambda$-bounded class of morphisms and $\cw$ a class of morphism of $\ck$ 
such that
\begin{enumerate}
\item[(1)] $\cw$ is a class-$\lambda$-accessible and strongly class-$\lambda$-accessibly embedded subcategory of $\ck^\to$ with the 2-out-of-3 property,
\item[(2)] $\ci^\square\subseteq\cw$, and
\item[(3)] $\cof(\ci)\cap\cw$ is closed under pushout and transfinite composition and cone-coreflective in $\ck^\to$.
\end{enumerate}
Then, taking $\cof(\ci)$ for cofibrations and $\cw$ for weak equivalences, we get a model category structure on $\ck$.
\end{theo}
\begin{proof}
Since $\ci$ is $\lambda$-bounded, $(\cof(\ci),\ci^\square)$ is a cofibrantly class-$\lambda$-ge\-ne\-ra\-ted weak factorization system.
For every $\lambda$-presentable $w\in\cw$, we construct a factorization in $\ck$ into a cofibration $j$ followed by a trivial fibration. By 2-out-of-3 
property for $\cw$ and (2), $j$ is in $\cw$. Let $\cj$ be the class of these morphisms $j$ for all $\lambda$-presentable $w\in\cw$.
 
We will check now the conditions of Lemma~\cite[1.8]{B}. We have to show that for every morphism $i\to w$ in $\ck^\to$ with $i\in \ci$ and $w\in\cw$ 
there exists $j\in \cj$ that factors it $i\to j\to w$. First note that there exists a $\lambda$-presentable $w'\in\cw$, which factors the original 
morphism, since every $w$ is a $\lambda$-filtered colimit of $\lambda$-presentable objects $\cw$ and every $i\in I$  is $\lambda$-presentable 
in $\ck^\to$; we used here that the inclusion of $\cw$ to $\ck^\to$ preserves $\lambda$-presentable objects. Next, decompose that morphism $w'$ 
into a cofibration $j\in\cj$ followed by a trivial fibration. The lifting axiom in $\ck$ finishes this argument.

Lemma~\cite[1.8]{B} implies that $\cof\cj= \cof\ci\cap\cw$. The requirement that $\cof\ci\cap\cw$ is cone-coreflective in $\ck^{\to}$ 
ensures that $\cj$ is cone-coreflective as well (by the same argument as above). By construction, the domains of all the elements in $\cj$ 
are $\lambda$-presentable. Hence $\cj$ satisfies the assumptions of \cite{CR} 4.3 and thus $(\cof(\cj),\cj^\square)$ is a weak factorization system.
Since $\cw$ is closed under retracts in $\ck^\to$ (cf. \cite{AR} 2.4 and 2.5), we get a model category structure on $\ck$. 
\end{proof} 

\begin{rem}\label{re2.14}
{\em
Let $\ck$ be a locally presentable category, $\ci$ a set of morphisms and $\cw$ a class of morphism of $\ck$
such that
\begin{enumerate}
\item[(1)] $\cw$ has the 2-out-of-3 property and is closed under retracts in $\ck^\to$,
\item[(2)] $\ci^\square\subseteq\cw$, and
\item[(3)] $\cof(\ci)\cap\cw$ is closed under pushout and transfinite composition.
\end{enumerate}
Then, taking $\cof(\ci)$ for cofibrations and $\cw$ for weak equivalences, we get a combinatorial model category
if and only if the inclusion of $\cw$ in $\ck^\to$ is accessible. This is the content of the Smith's theorem 
(see \cite{B} for sufficiency and \cite{L} or \cite{R2} for necessity). 

We do not know whether this can be generalized to class-accessible setting and \ref{th2.13} is what we are able to do.  
The question is whether cone-coreflectivity of $\cof(\ci)\cap\cw$ follows from the other assumptions. We also do not know whether
the model category in \ref{th2.13} is class-combinatorial. Indeed, we only know that the class $\cj$ satisfies \cite{CR} 4.7 (1). 
}
\end{rem}
 
Homotopy equivalences can be defined in any category $\ck$ with finite coproducts which is equipped with a weak 
factorization system $(\cl,\crr)$ (see \cite{R2}). Recall that a \textit{cylinder object} $C(K)$ of an object $K$ 
is given by an $(\cl,\crr)$ factorization of the codiagonal
$$
\nabla : K+K \xrightarrow{\quad  \gamma_K\quad} C(K)
             \xrightarrow{\quad \sigma_K\quad} K
$$
We denote by
$$
\gamma_{1K},\gamma_{2K}:K\to C(K)
$$ 
the compositions of $\gamma_K$ with the coproduct injections. 

Then, as usual, we say that morphisms $f,g:K\to L$ are \textit{homotopic}, and write $f\sim g$, if there is a morphism 
$h:C(K)\to L$ such that the following diagram commutes
$$
\xymatrix@=3pc{
K+K \ar[rr]^{(f,g)}
\ar[dr]_{\gamma_K} && L\\
& C(K) \ar[ur]_h
}
$$
Here, $(f,g)$ is induced by $f$ and $g$. The homotopy relation  $\sim$ is clearly reflexive, symmetric, 
compatible with the composition and does not depend on the choice of a cylinder object. But, it is not transitive 
in general and we will denote its transitive hull by $\approx$. We get the quotient functor
$$
Q:\ck\to\ck/\approx.
$$

A morphism $f:K\to L$ is called a \textit{homotopy equivalence} if $Qf$ is the isomorphism, i.e., if there exists
$g:L\to K$ such that both $fg\approx\id_L$ and $gf\approx\id_K$. The full subcategory of $\ck^\to$ consisting
of homotopy equivalences w.r.t. a weak factorization system $(\cl,\crr)$ will be denoted by $\ch_\cl$. The following
result generalizes \cite{R2}, 3.8.

\begin{propo}\label{prop2.15}
Let $\ck$ be a locally class-presentable category and $(\cl,\crr)$ be a weak factorization system with a strongly class-accessible cylinder
functor. Then $\ch_\cl$ is a full image of a strongly class-accessible functor into $\ck^\to$.
\end{propo}  

\begin{proof}
Given $n<\omega$, let $\cm_n$ be the category whose objects are $(4n+2)$-tuples
$$
\tau =(f,g,a_1,\dots,a_n,b_1,\dots,b_n,h_1,\dots,h_n,k_1,\dots,k_n)
$$
of morphisms $f:A\to B$, $g:B\to A$, $a_1,\dots,a_n:A\to A$, $b_1,\dots,b_n:B\to B$, $h_1,\dots,h_n:C(A)\to A$
and $k_1,\dots,k_n:C(B)\to B$. Morphisms are pairs $(u,v)$ of morphisms $u:A\to A'$ and $v:B\to B'$ such that
$f'u=vf$, $g'v=ug$, $uh_i=h'_iC(u)$ and $vk_i=k'_iC(v)$ for $i=1,\dots,n$. This category is obtained by an inserter 
construction inserting our $n+2$ morphisms among $\Id$ and $C$. Since $C$ is strongly class-accessible, the procedure 
of the proof of \cite{CR} 3.9 yields that $\cm_n$ is a class-accessible category. 

Let $\overline{\cm}_n$ be the full subcategory of $\cm_n$ such that $h_1\gamma_A=(gf,a_1)$, $h_i\gamma_A=(a_i,a_{i+1})$, 
$h_n\gamma_n=(a_n,\id_A)$, $k_1\gamma_A=(fg,b_1)$, $k_i\gamma_A=(b_i,b_{i+1})$ and $k_n\gamma_n=(b_n,\id_B)$ where $1<i<n$. 
This category is obtained from $\cm_n$ by an equifier construction and, by the same reason as above, the procedure
of the proof of \cite{CR} 3.7 yields that $\overline{\cm}_n$ is class-accessible.  Moreover, its inclusion into $\cm_n$
is strongly accessible.

We have full embeddings
$$
M_{m,n}:\cm_m\to\cm_n,
$$
for $m<n$, which take the missing $a_i,b_i,h_i,k_i $ as the identities. The union $\cm$ of all $\cm_n$'s is a class-accessible 
category. Since all $\overline{\cm}_n$'s are strongly accessibly embedded to $\cm_n$, $\overline{\cm}$ is strongly accessible 
embedded by to $\overline{\cm}$. Let 
$$
F:\overline{\cm}\to\ck^\to
$$
sends each $(4n+2)$-tuple above to $f$. This is a strongly class-accessible functor whose image is $\ch_\cl$.
\end{proof}

\section{Left Bousfield localizations}
Recall that $\tilde h$ is a \textit{cofibrant approximation} of $h$ if there is a commutative square
$$
\xymatrix@=4pc{
A \ar [r]^{v} \ar [d]_{h}& \tilde{A} \ar [d]^{\tilde{h}}\\
B\ar [r]_{w}& \tilde{B}
}
$$
where $v$ and $w$ are weak equivalences. 

\begin{defi}\label{def3.1}
{
\em
Let $\ck$ be a class-combinatorial simplicial model category and $\cal F$ a class of morphisms of $\ck$. Assume 
that $\cal F$ contains only cofibrations between cofibrant objects. An object $K$ in $\ck$ is called 
$\cal F$-\textit{local} if it is fibrant and
$$
\hom(f,K)\colon\hom(B,K)\to\hom(A,K)
$$
is a weak equivalence of simplicial sets for each $f\colon A\to B$ in $\cf$.

A morphism $h$ of $\ck$ is called an $\cf$-\textit{local equivalence} if $\hom(\tilde h,K)$ is a weak equivalence for each 
$\cf$-local object $K$; here, $\tilde h$ is a cofibrant approximation of $h$. 

The full subcategory of $\ck$ consisting of $\cf$-local objects is denoted $\Loc(\cf)$ and the full subcategory 
of $\ck^\to$ consisting of $\cf$-local equivalences is denoted $\LEq(\cf)$.

We say that there exists a \textit{left Bousfield localization} of $\cf$ if cofibrations in $\ck$ and $\cf$-local equivalences
form a model category structure on $\ck$. 
}
\end{defi}

\begin{rem}\label{re3.2}
{
\em
(1) It is easy to see that the definition of a local $\cf$-equivalence does not depend on the choice of a cofibrant approximation.

(2) Following \cite{H}, 9.3.3 (2), any weak equivalence in $\ck$ is an $\cf$-local equivalence. On the other hand, every
$\cf$-local equivalence between $\cf$-local objects is a weak equivalence in $\ck$ (cf. \cite{GJ} X.2.1. 2)).

(3) If $\ck$ is left proper then the intersection of cofibrations and $\cf$-local equivalences is closed under pushout and transfinite 
composition (see \cite{H}, 13.3.10, 17.9.4 for a non-trivial part of the proof); the trivial part is that $\hom(-,K)$ sends colimits
to limits and cofibrations to fibrations. It is also closed under retracts in $\ck^\to$ of course.
}
\end{rem}
  
Given a morphism $f$, $\{f\}$-local objects are called $f$-local and analogously for $f$-local equivalences.
The corresponding categories are called $\Loc(f)$ and $\LEq(f)$.

\begin{propo}\label{prop3.3}
Let $\ck$ be a class-combinatorial simplicial model category and $\cf$ a set of cofibrations between cofibrant 
objects of $\ck$. Then  $\Loc(\cf)$ is a class-accessible category strongly accessibly embedded in $\ck$.
\end{propo}
\begin{proof}
$$
\xymatrix@=4pc{
\ck \ar[r]^{\hom(f,-)} & \SSet^\to \\
\Loc(f) \ar [u]^{} \ar [r]_{} &
\cw \ar[u]_{}
}
$$
is a pullback where $\cw$ denotes weak equivalences in $\SSet$. Since the vertical leg on the right is transportable, $\Loc(f)$ is a pseudopullback
and thus it is a class-accessible and its inclusion to $\ck$ is strongly class-accessible (see \cite{CR} 3.1, 3.2 and \ref{th2.10}). Since 
$$
\Loc(\cf)=\bigcap\limits_{f\in\cf} \Loc(f),
$$ 
the result follows from \cite{CR} 3.3.
\end{proof}

\begin{defi}\label{def3.4}
{
\em
Let $\ck$ be a simplicial model category and $f:A\to B$ a cofibration of cofibrant objects. Consider a pushout
$$
\xymatrix@=4pc{
\partial\Delta_n\otimes A \ar [r]^{\id\otimes f} \ar [d]_{i_n\otimes\id}& \partial\Delta_n\otimes B \ar [d]^{p_{n1}}\\
\Delta_n\otimes A\ar [r]_{p_{n2}}& P_{f,n}
}
$$
where $i_n:\partial\Delta_n\to\Delta_n$ is the inclusion of the boundary into a simplex. Let 
$h_{f,n}:P_{f,n}\to \Delta_n\otimes B$ be the canonical morphism, which is a cofibration since $\ck$ is simplicial. 

Cofibrations $h_{f,n}$, $n=0,1,\dots$ are called $f$-\emph{horns}. If $\cf$ is a class of cofibrations, then 
we denote by $\Hor(\cf)$ the collection of all $f$-horns, for all $f\in \cf$.
}
\end{defi}

\begin{rem}\label{re3.5}
{
\em
Every $h_{f,n}\in \Hor(\cf)$ is an $\cf$-local equivalence because the morphism 
$$
\hom(h_{f,n},K):\hom(\Delta_{n}\otimes B,K)\to \hom(P_{f,n},K)
$$
is a weak equivalence for every $\cf$-local object $K$. In fact, the morphism 
$$
\hom(\id\otimes f,K):\hom(\Delta_n\otimes B,K)\to\hom(\Delta_n\otimes A,K)
$$ 
is a weak equivalence because $K$ is $\cf$-local and $\hom(p_{n,2},K)$ is a trivial fibration as a pullback of the trivial fibration 
$$
\hom(\id\otimes f,K):\hom(\partial\Delta_n\otimes B,K)\to \hom(\partial\Delta_n\otimes A,K).
$$
Thus it suffices to use the 2-out-of-3 property.

We used the fact that $f$-horns are cofibrations between cofibrant objects and that the definition of an $\cf$-local equivalence does not depend
on the choice of a cofibrant approximation.
}
\end{rem}

\begin{lemma}\label{le3.6}
Let $\ck$ be a class-combinatorial simplicial model category and $\cf$ a class of cofibrations between cofibrant
objects of $\ck$. Then a fibrant object $K$ of $\ck$ is $\cf$-local if and only if it is injective to all $f$-horns 
for $f\in\cf$.
\end{lemma}
\begin{proof}
Since each $f\in\cf$ is a cofibrations, $\hom(f,K)$ is a fibration for each fibrant object $K$. Thus 
a fibrant object $K$ is $\cf$-local if and only if $\hom(f,K)$ is a trivial fibration for each $f\in\cf$. 
This is the same as having the right lifting property with respect to each inclusion $i_n:\partial\Delta_n\to\Delta_n$. 
The latter is clearly equivalent to being injective with respect to $f$-horns $h_{f,n}$ for all $f\in\cf$.
\end{proof} 

\begin{lemma}\label{le3.7}
Let $\cf$ be a cone-coreflective class of cofibrations between $\lambda$-presentable cofibrant objects. Then $\Hor(\cf)$ 
is cone-coreflective class of morphisms between $\lambda$-presentable objects.
\end{lemma}
\begin{proof}
Since $\partial\Delta_n\otimes B$ and $P_{f,n}$ are $\lambda$-presentable provided that $A$ and $B$ are $\lambda$-presentable,
we have to prove that $\Hor(\cf)$ is cone-coreflective. Let $f\colon A\to B$ be an element of $\cf$. Given a commutative square
$$
\xymatrix@=4pc{
P_{f,n} \ar [r]^{} \ar [d]_{h_{f,n}}& X \ar [d]^{g}\\
\Delta_n\otimes B\ar [r]_{}& Y
}
$$
with $h_{f,n}\in \Hor(\cf)$ and $g$ arbitrary, we form, by adjunction, the following commutative square:
\[
\xymatrix@=4pc{
A \ar[d]_{f} \ar[r] & X^{\Delta^{n}} \ar[d]^{g'}\\
B \ar[r] & Q_{g,n}
}
\]
where $Q_{g,n} = X^{\partial\Delta^n}\times_{Y^{\partial\Delta^n}} Y^{\Delta^{n}}$.

Since $\cf$ is cone-coreflective, there exists a set of morphisms $\cf_{g'}=\{f'\colon A'\to B'\}\subset \cf$ such that any morphism $f\to g'$ 
in $\ck^\to$ factors through some element $f'\in\cf_{g'}$. Unrolling back the adjunction, we obtain the set of horns 
$\Hor(\cf_{g'})=\{h_{f',n}: P_{f',n}\to\Delta^n\otimes B' \;|\, n\geq 0\}$ which depends entirely on $g$. Thus $\Hor(\cf)$
is cone-coreflective. 
\end{proof}

\begin{rem}\label{re3.8}
{
\em
(1) Let $\cf$ be a cone-coreflective class of cofibrations between $\lambda$-presentable cofibrant objects in a class $\lambda$-combinatorial
simplicial model category $\ck$. Then $\Loc(\cf)$ is weakly reflective and closed under $\lambda$-filtered colimits in $\ck$ (following \ref{le3.7}, 
\ref{le3.6} and \cite{CR} 4.4). Recall that a weak reflection $r_K:K\to K^\ast$ is obtained as a factorization 
$$
K \xrightarrow{\quad  r_K\quad} K^\ast\xrightarrow{\quad \quad} 1.
$$
in $(\cof(\Hor(\cf)\cup\cc),(\Hor(\cf)\cup\cc)^\square)$ where $\cc$ is a bounded class such that $\cof(\cc)$ are cofibrations in $\ck$. Thus 
$r_K$ belongs to $\cof(\Hor(\cf)\cup\cc)$. If $\ck$ is left proper then, following \ref{re3.2} (3), 
$\cof(\Hor(\cf)\cup\cc)\subseteq\cof(\cc)\cap\LEq(\cf)$. Hence $r_K$ is both a cofibration and an $\cf$-local equivalence. 

But this does not mean that weak reflections are functorial, i.e., that there exists a functor $L:\ck\to\Loc(\cf)$ and a natural transformation 
$\eta:\Id\to L$ such that $\eta_K=r_K$ for each $K$ in $\ck$. Such a functor $L$ is called an $\cf$-\textit{localization functor}.

(2) Given a model category $\ck$ and a functor $L:\ck\to \ck$, then $\LEq(L)$ will denote the class of morphisms sent by $L$ to weak equivalences.
If both the left Bousfield localization and a localization functor $L$ exist for $\cf$, then $\LEq(\cf)=\LEq(L)$. 

In fact $h$ is an $\cf$-local equivalence iff its cofibrant approximation $\tilde{h}$ is an $\cf$-local equivalence. Since $\eta_K$ is
an $\cf$-local equivalence for each $K$, $\tilde{h}$ is an $\cf$-local equivalence iff $L(\tilde{h})$ is an $\cf$-local equivalence,
i.e., a weak equivalence in $\ck$ (see \ref{re3.2} (2)).
}
\end{rem}

\begin{propo}\label{prop3.9}
Let $\ck$ be a nice class-combinatorial model category and $L:\ck\to \ck$ be a strongly class-accessible functor. Then $\LEq(L)$ 
is a class-accessible category strongly class-accessibly embedded in $\ck^\to$.
\end{propo}
\begin{proof}
By assumption, the class $\cw$ of weak equivalences is class-acce\-ssib\-le and  strongly class-accessibly embedded in $\ck^\to$. Since
$\LEq(L)$ is given by the pullback
$$
\xymatrix@=4pc{
\LEq(L) \ar [r]^{} \ar [d]_{}& \cw \ar [d]^{}\\
\ck^\to\ar [r]_{L^\to}& \ck^\to
}
$$
having the vertical leg on the right transportable, $\LEq(L)$ is a pseudopullback and thus class-accessible and strongly class-accessibly embedded 
in $\ck^\to$ (see \cite{CR} 3.1 and 3.2).
\end{proof} 

\begin{theo}\label{th3.10}
Let $\ck$ be a nice, class-combinatorial, left proper, simplicial model category and let $\cf$ be a class of morphisms in $\ck$. Suppose there exists 
a strongly class-accessible $\cf$-localization functor $L:\ck\to\ck$. Then the left Bousfield localization of $\ck$ with respect to $\cf$ exists 
if and only if the intersection of $LEq(\cf)$ with the cofibrations of $\ck$ is a cone-coreflective class of morphisms.
\end{theo}
\begin{proof}
Necessity immediately follows from the existence of the (trivial cofibration, fibration) factorizations in the localized model category cf.
\cite{CR} 4.2 (2)).
  
In order to establish sufficiency, we will verify the conditions of \ref{th2.13}. By \ref{re3.8} (2) and \ref{prop3.9} the subcategory $\LEq(\cf)$ 
is class-accessible. There is a regular cardinal $\lambda$ such that $\LEq(\cf)$ is class-$\lambda$-accessible and $\ck$ is class-$\lambda$-combinatorial. 
In fact, it $\LEq(\cf)$ is class-$\mu$-accessible and $\ck$ is class-$\nu$-combinatorial, it suffices to take $\mu,\nu\vartriangleleft\lambda$.
Let $\ci$ be the generating class of cofibrations in $\ck$. Then  $\ci^\square\subseteq\LEq(\cf)$ because $\LEq(\cf)$ contains all weak equivalences. 
Following \ref{re3.2} (3), $\cof\ci\cap \LEq(\cf)$ is closed under pushouts and transfinite compositions.
\end{proof} 

\begin{exam}\label{ex3.11}
{
\em
Let $\ca$ be a complete simplicial category. Then $\cp(\ca)$ equipped with the projective model structure is a class-combinatorial model 
category (see \ref{prop2.4}). Let $f:V\to W$ be a cofibration of simplicial sets. Then the class  $\cf = \{f\otimes\hom(-,A) \, | \, A\in \ca\}$ 
is bounded. The argument is the same as in the proof of \ref{prop2.4}. The localization of $\cp(\ca)$ with respect to $\cf$ is equivalent 
to the levelwise localization with respect to $f$. 

Let $L_f:\SSet\to\SSet$ be the $f$-localization functor, i.e., a fibrant replacement functor in the $f$-localized model category structure
on $\SSet$. Then $L_f$ is finitely accessible provided that $V$ and $W$ are finitely presentable.
Moreover, $L_f$ is always simplicial (see \cite{S}, 24.2). Similarly to Lemma \ref{le2.6}, we get a strongly class-accessible simplicial
functor $L:\cp(\ca)\to\cp(\ca)$ assigning to $F$ the composition $L_fF$. Since $\LEq(L)=\LEq(\cf)$, $\LEq(\cf)$ is a class-accessible 
subcategory of $\cp(\ca)^{\to}$ by \ref{prop3.9}.

In a general case, $L_f$ is accessible and we would need an extension of \ref{le2.3} to $\lambda$-small weighted limits. This is valid but
we have not burdened our paper with a proof. 
}
\end{exam}

\begin{rem}\label{re3.12}
{
\em
Let $\cf$ be a set of cofibrations between cofibrant objects in a nice class-combinatorial left proper model category $\ck$ such that $\ck$
admits a strongly class-accessible fibrant replacement functor and $\Hor(\cf)$ is bounded. Since $\Hor(\cf)$ is a set, there is a strongly
class-accesible weak reflection on $\Hor(\cf)$-injective objects (see \cite{CR} 4.8 (1)). We can assume that the both functors are strongly
class-$\lambda$-accessible (see \cite{CR} 2.8). Thus they are strongly class-$\lambda^+$-accessible (\cite{CR} 2.8 again) and, following
\cite{CR} 4.8 (5), there is a strongly class-$\lambda^+$-accessible $\cf$-localization functor $L$. Since $\LEq(\cf)=\LEq(L)$, $\LEq(\cf)$
is strongly class-accessible and strongly class-accessibly embedded in $\ck^\to$. Following \ref{th2.13}, the left Bousfield localization
of $\cf$ exists provided that the intersection of cofibration with $\LEq(\cf)$ is cone-coreflective. 
}
\end{rem}

Let $\ck$ be a model category. A functor $L:\ck\to \ck$ a equipped with natural transformation $\eta: \Id_{\ck}\to L$ is called
\textit{homotopy idempotent} if $L\eta_K$ and $\eta_{LK}$ are weak equivalences for each $K$ in $\ck$.
 
\begin{defi}\label{def3.13}
{
\em
Let $\ck$ be a model category equipped with a homotopy idempotent functor $L:\ck\to \ck$ preserving weak equivalences. A left Bousfield localization 
of $\ck$ with respect to $L$, or just $L$-\textit{localization} of $\ck$ is a new model structure on $\ck$ such that the class of cofibrations coincides 
with the original class of cofibrations in $\ck$ and the class of weak equivalences is $\LEq(L)$. New fibrations are called $L$-\textit{fibrations}.
}
\end{defi}

\begin{theo}\label{th3.14}
Let $\ck$ a nice, proper, simplicial class-combinatorial model category and $L:\ck\to\ck$ a strongly class-accessible homotopy idempotent functor 
preserving weak equivalences. Suppose additionally, that pullbacks of $L$-equivalences along $L$-fibrations are $L$-equivalences. Then 
the $L$-localization exists and is class-combinatorial.
\end{theo}
\begin{proof}
It was shown in \cite[Appendix A]{BF}, that the pair 
$$
(\cof(\ci)\cap\LEq(L),(\cof(\ci)\cap\LEq(L))^\square)
$$ 
is a weak factorization system. They argue as follows. 

Take $i\in \cof(\ci)\cap\LEq(L)$ and $f:X\to Y$. For any morphism $i\to f$ in $\ck^\to$ we perform the following construction:
\[
\xymatrix{
A\ar[r] \ar@{^(->}[ddd]_i & X\ar[rrr] \ar[ddd]_{f} \ar@{^(->}[dr] &    &   & LX\ar@{^(->}[dd]^{\dir{~}}\\
     &     &      Z \ar@{->>}[dr]^{\dir{~}}\\
     &     &             &P\ar[r] \ar@{->>}[dll] & W \ar@{->>}[d]\\
B\ar[r] \ar@{-->}[uurr]& Y \ar[rrr]&    &     &L
Y.\\
}
\]
After applying the functor $L$ on the morphism $f$ we factor $Lf$ as a trivial cofibration followed by a fibration in $\ck$, obtaining the $L$-fibration 
$W\to LY$, since this is a fibration of $L$-local objects. Then constructing $P= W\times_{LY} Y$ we obtain an $L$-fibration $P\to Y$ as a pullback 
of an $L$-fibration and an $L$-equivalence $P\to W$ due to the additional assumption. The induced morphism  $X\to P$ is an $L$-equivalence 
by the 2-out-of-3  property. Now we factor the morphism $X\to P$ into a cofibration followed by a trivial fibration in $\ck$. As the composition 
of two $L$-fibrations, the morphism $Z\to Y$ is an $L$-fibration, hence there exists a lift $B\to Z$, showing that $\cof(I)\cap\LEq(\cf)$ 
is cone-coreflective. 

Like in the proof of \ref{th3.10}, there exists a regular cardinal $\lambda$ such that $\ck$ is $\lambda$-combinatorial and $L$ strongly 
class-$\lambda$-accessible. Assume  that $X$ and $Y$ are $\lambda$-presentable. Then $LX, LY$ and $W$ are $\lambda$-presentable. Since $\ck$ 
is locally $\lambda$-presentable simplicial category, the functor $E:\ck\to\cp(\ca)$ from the proof of \cite{CR} 2.6 takes values in simplicial 
presheaves and thus $\cp(\ca)$ can be taken in the sense of \ref{ex2.2}. The functor $E$ sends $\lambda$-presentable objects to finitely presentable 
ones and thus it is strongly class-$\lambda$-acccessible. Since $E$ preserves limits (see \cite{CR} 2.6), \ref{le2.3} implies that $EP$ is 
$\lambda$-presentable. Since $\ck$ is closed under $\lambda$-filtered colimits in $\cp(\ca)$, $P$ is $\lambda$-presentable. Thus $Z$ is 
$\lambda$-presentable. Consequently, the $L$-localized model category is class-$\lambda$-combinatorial. 
\end{proof}

\begin{exam}\label{ex3.15}
{
\em
Take $f:V\to 1$ in Example \ref{ex3.11}. For such maps $f$-localization functor is called also $V$-nullification. Then the resulting class 
of $f$-equivalences satisfies the conditions of \ref{th3.14}, since the nullification of spaces (i.e., the localization with respect to $f$ 
for $f$ as above) is a right proper model category (see, e.g., \cite{Bo1}). Hence the model category resulting from the levelwise nullification 
of the projective model structure on the category of small functors is class-combinatorial again. 
}
\end{exam}

\begin{exam}\label{ex3.16}
{\em
Consider the category $\cp(\SSet^{\op})$ of small simplicial functors from simplicial sets to simplicial sets equipped with the projective model 
structure (see \ref{prop2.4}). Consider the localization functor $L:\cp(\SSet^{\op})\to\cp(\SSet^{\op})$, $L=P_{n}\circ \Fib$ constructed 
in \cite{BCR}, where $P_{n}$ is Goodwillie's $n$-th polynomial approximation \cite{Goodwillie-CalculusIII} and $\Fib:\cp(\SSet^{\op})\to\cp(\SSet^{\op})$ 
is the strongly class-accessible fibrant replacement functor from \ref{le2.6}. Since $P_{n}$ is a countable colimit of finite homotopy limits of cubical 
diagrams applied on homotopy pushouts (joins with finite sets used to construct $P_n$ in \cite{Goodwillie-CalculusIII} may be expresses as homotopy 
pushouts), it is strongly class-accessible. Thus $L$ is strongly class-accessible, hence the polynomial model structure constructed in \cite{BCR} 
is class-combinatorial. 
}
\end{exam}

The condition on the localization functor to be strongly class-acce\-ssib\-le may not be omitted in \ref{th3.14} as the following example shows. 

\begin{exam}\label{ex3.17}
{
\em
The following localization of the class-combinatorial model category $\cp(\SSet)$ was constructed in \cite{C1}. The localization functor 
$L:\cp(\SSet)\to\cp(\SSet)$ is the composition of the evaluation functor at the one point space $\ev_\ast(F)= F(1)$ with the fibrant replacement 
$\widehat{(-)}$ in simplicial sets and the Yoneda embedding $Y:\SSet\to\cp(\SSet)$, i.e., $L(F) = \hom(-, \widehat{F(1)})$. This localization functor 
satisfies the conditions of $A6$ in \cite{BF} (pullback of an $L$-equivalence along an $L$-fibration is an $L$-equivalence again), 
and hence there exists the $L$-local model structure on $\cp(\SSet)$. The fibrant objects in the localized model category are the levelwise fibrant 
functors weakly equivalent to the representable functors, but they are not closed under filtered colimits, since filtered colimit of representable 
functors need not be representable, no matter how large the filtered colimit is. On the other hand, in a class-cofibrantly generated model category 
sufficiently large filtered colimits of fibrant objects are fibrant again. In other words, we obtained the a localization of $\cp(\SSet)$, which is not 
class-cofibrantly generated. The reason is that the localization functor $L$ is not class-accessible. See \cite{C1} for more details on this model 
structure.
}
\end{exam}

\end{document}